\documentclass[12pt]{article}

\usepackage{amsmath,amsthm}
\usepackage{amsfonts,xcolor,graphicx,url}
\usepackage{epstopdf,epsfig}

\usepackage[margin=2.5cm]{geometry}

\usepackage{titlesec}
\titleformat{\section}
{\normalfont\bfseries}{\thesection}{1em}{}
\titleformat{\subsection}
{\normalfont}{\thesubsection}{1em}{}

\numberwithin{equation}{section}

\newcommand{\euc}[1]{\mathbb{E}^{#1}}

\newcommand{\tens}{\mathrm}
\renewcommand{\vec}{}

\theoremstyle{plain}
\newtheorem{theorem}{Theorem}
\newtheorem{lemma}[theorem]{Lemma}

\newtheorem{corollary}[theorem]{Corollary}
\theoremstyle{definition}
\newtheorem{definition}[theorem]{Definition}

\newtheorem{example}[theorem]{Example}

\theoremstyle{remark}
\newtheorem{remark}{Remark}

\begin{document}

\title{Piecewise Flat Extrinsic Curvature}

\author{Rory Conboye
	}
\AtEndDocument{\raggedleft {\sc \ \\
Department of Mathematics and Statistics, \\ American University, \\ Washington, DC 20016} \\ \ \\ conboye@american.edu}





\date{}
\maketitle

\begin{abstract}


Discretizations of the mean curvature and extrinsic curvature components are constructed on piecewise flat simplicial manifolds, giving approximations for smooth curvature values in a mostly mesh-independent way. These constructions are given in combinatoric form in terms of the extrinsic hinge angles, the intrinsic structure of the piecewise flat manifold and a choice of dual tessellation, and can be viewed as the average of $n$-volume integrals. The constructions are also independent of the manifold dimension.


\




\end{abstract}


\section{Introduction}


The extrinsic curvature describes the shape of a manifold due to its' embedding in some higher dimensional manifold. For example the curvature of a smooth curve in Euclidean $2$-space $\euc{2}$ is defined as the inverse of the radius of the circle that best fits the curve at each point. Piecewise flat simplicial manifolds are generalizations to higher dimensions of piecewise linear curves, and are formed by joining Euclidean $n$-simplices (line-segments, triangles, tetrahedra) along their $(n-1)$-dimensional faces. Due to this piecewise flat nature the curvature cannot be given in the same way as smooth manifolds. Instead, the shape resulting from an embedding is determined by \emph{hinge-angles} defined at the faces separating each pair of $n$-simplices. This is easily visualised for a piecewise linear curve in $\euc{2}$, with an obvious angle between each pair of line-segments.


This piecewise flat interpretation of the extrinsic curvature dates back to Steiner in the 1840's \cite{Steiner}, where the total mean extrinsic curvature of a tetrahedron is given as the sum for each edge, of $\pi$ minus the exterior dihedral angle, times the length of that edge. This approach is similar to the deficit angles around co-dimension-$2$ simplices which give the intrinsic curvature of a piecewise flat manifold, with the related total sum giving the Regge action \cite{Regge}.


Piecewise flat manifolds can provide a discrete approximation for smooth manifolds in a number of ways, and since they are geometric objects themselves, they provide many advantages over other approximation techniques. However, while hinge angles are appropriate for piecewise flat manifolds, they do not measure curvature in the same way as smooth extrinsic curvature. What they \emph{do} measure, is equivalent to the path integral of certain curvature components across each hinge. Since the choice of path can be ambiguous, dividing these integrals by a path length does not lead to a consistent definition.


In this paper the curvature is instead averaged over a specific collection of $n$-volumes. These $n$-volumes are given by a choice of dual tessellation and defined so as to intersect a collection of hinges appropriate to either the mean curvature at a vertex, or the curvature component orthogonal to a hinge. The $n$-volume integral is constructed using path integrals across each hinge, given in terms of the hinge angles, and divided by the $n$-volume to give an average value. This is an extension of an approach developed for \emph{intrinsic} piecewise flat curvature \cite{PLCurv}, which has lead to a piecewise flat Ricci flow in three dimensions.


Previous approaches for approximating smooth extrinsic curvatures have mostly been for $2$-dimesnional surfaces in $\euc{3}$. One of the more prominent of these makes use of the \emph{cotan} formula of Pinkall and Polthier \cite{PP93}, giving the integrated mean curvature over areas dual to vertices as a weighted sum of the edge vectors in $\euc{3}$, see for example Meyer et al. \cite{MeyDebSchBar03} and a discussion about convergence by Wardetzky \cite{WarDDG}. However as noted in \cite{BS07}, non-Delaunay triangulations lead to negative weightings, which can cause certain issues to arise. While methods exist to adapt a given triangulation to Delaunay for $2$-surfaces, this is more difficult for higher dimensions where Delaunay triangulations also become more restrictive \cite{BDGintrins}. Other approaches have used weighted sums of the hinge angles bounding individual triangles \cite{Grin06}, and area variations for families of parallel polyhedral meshes \cite{BobCurv10}, while methods related to Regge calculus in the physics literature mostly concentrate on single hinges \cite{HS81, Brewin, KLM89ec, TraceK}.

The main results of the paper are highlighted in theorem \ref{thm:1.1} below, giving the piecewise flat integrated curvature along a geodesic segment, the average mean curvature at a vertex, and the average extrinsic curvature component orthogonal to each hinge. While the choice of dual tessellation is left open for the remainder of the paper, for simplicity it is restricted to tessellations which intersect each edge at its' midpoint below. Such tessellations include Voronoi, barycentric or the mixture area used in \cite{MeyDebSchBar03}.

\begin{theorem}
\label{thm:1.1} Take an $n$-dimensional piecewise flat manifold $S^n$, embedded in $\euc{n+1}$.

\begin{enumerate}
\item For a geodesic segment $\gamma$ in $S^n$, intersecting a single hinge $h$ at an angle $\theta$ to the vector orthogonal to $h$ in $S^n$, the integral of the second fundamental form $\alpha$ along $\gamma$ is
\begin{equation}
\int_\gamma \alpha (\hat \gamma, \hat \gamma) \, \mathrm{d} s
= \cos \theta \ \epsilon_h + O(\epsilon_h^3) ,
\end{equation}
for small hinge angle $\epsilon_h$, with $\hat \gamma$ giving the unit tangent vector field to $\gamma$.

\item The average mean curvature over the dual $n$-volume $V_v$ at a vertex $v$ is
\begin{equation}
H_v
:= \widetilde H_{V_v}
= \frac{1}{|V_v|}
\sum_{h \subset \mathrm{star}(v)} \frac{1}{2}|h| \epsilon_h ,
\label{1.H}
\end{equation}
with $|V_v|$ and $|h|$ representing the $n$-volume measures of $V_v$ and $h$ respectively.

\item The average curvature orthogonal to $h$ over an $n$-volume $V_h$ (formed by the $n$-volumes dual to the vertices in the closure of $h$, intersecting lines orthogonal to $h$ in $S^n$), is
\begin{equation}
\alpha_h
:= \widetilde{\alpha}(\hat \gamma, \hat \gamma)_{V_h}
= \frac{1}{|V_h|} \left(
|h| \epsilon_h
+ \sum_{i} \frac{1}{2} |h_i| \cos^2 \theta_i \, \epsilon_i
\right) ,
\label{1.a}
\end{equation}
for hinges $h_i$ intersecting $V_h$, making an angle $\theta_i$ with $h$, with hinge angles $\epsilon_i$.

\end{enumerate}
\end{theorem}


The paper begins with some preliminaries in section \ref{sec:Prelim}, giving definitions for piecewise flat manifolds, hinge angles and triangulations of smooth embedded manifolds. The integral of the extrinsic curvature over hinge-intersecting paths is then found in section \ref{sec:Int}. The following two sections motivate the choice of volumes over which to compute the mean curvature and hinge-orthogonal components, and prove expressions (\ref{1.H}) and (\ref{1.a}) for these volumes. The curvature values for triangulations of a circle in $\euc{2}$ and a $2$-sphere and $2$-cylinder in $\euc{3}$ are computed as examples throughout the paper, showing the constructions to match closely with their smooth curvature values. Generalizations to embeddings in higher dimensions is discussed briefly in section \ref{sec:GenEmb}.


\section{Piecewise Flat Manifolds and Submanifolds}
\label{sec:Prelim}

\subsection{Piecewise flat manifolds}

A piecewise flat manifold $S^n$ is a differential manifold which can be decomposed into Euclidean segments, with the most simple given by $n$-simplices (line-segments, triangles, tetrahedra). The geometry of a Euclidean $n$-simplex is entirely determined by the lengths of its edges, and so a simplicial piecewise flat manifold is also entirely determined by its simplicial graph and set of edge-lengths. 

\begin{definition}[Piecewise flat manifolds]
\label{def:PL}

\

\begin{enumerate}
\item A Euclidean $k$-simplex is an open subspace of a $k$-dimensional Euclidean space $\euc{k}$, given by the interior of the convex hull formed by $k + 1$ non-colinear points. An arbitrary $k$-simplex is denoted $\sigma^k$, with its closure denote $\bar \sigma^k$.


\item A homogeneous simplicial $n$-complex is a connected graph of simplices up to dimension $n$, where each simplex is either an $n$-simplex or a face of an $n$-simplex.


\item A piecewise flat manifold $S^n$ is a homogeneous simplicial $n$-complex formed by Euclidean $n$-simplices, such that the metric for each $k$-simplex $\sigma^k$ is consistent with the metrics of all $n$-simplices containing $\sigma^k$ in their closure.


\item The $\mathrm{star}$ of a $k$-simplex $\sigma^k$ is the subspace of a piecewise flat manifold $S^n$ formed by the set of simplices $\sigma_i^m$ containing $\sigma^k$ in their closures, i.e. $\mathrm{star}(\sigma^k) = \{ \sigma_i^m | \bar \sigma_i^m \supset \sigma^k \}$. In two dimensions this is comonly referred to as a $1$-ring.
\end{enumerate}
\end{definition}

The development from one $n$-simplex to an adjacent $n$-simplex is still Euclidean, with deviations from a Euclidean manifold arising for closed paths around co-dimension-two simplices. The parallel transport of a vector around such a path deviates by an \emph{intrinsic} deficit angle associated with the co-dimension-two simplex, in the plane orthogonal to it. The intrinsic curvature of a piecewise flat manifold is characterised by these intrinsic deficit angles, which can be derived from the simplicial graph and set of edge-lengths, much like smooth curvature can be given in terms of the metric.

\subsection{Hinge angles}



Any pair of $n$-simplices joined along a common $(n-1)$-face can be linearly embedded in $\euc{n+1}$ so that each $n$-simplex is flat. The common $(n-1)$-simplex is known as a \emph{hinge}, denoted $h$, and the relative directions of the two $n$-simplices is determined by a \emph{hinge angle} $\epsilon_h$ in the plane orthogonal to $h$ in $\euc{n+1}$. There are many, mostly equivalent, methods for defining the hinge angle, see for example \cite{SulDDG,KLM89ec}. One such method uses the angle $\phi$ between the positive normal vector $\vec{n}$ to one of the $n$-simplices, from a given orientation, and the inward-directed vector $\vec{v}$ tangent to the other $n$-simplex and orthogonal to $h$. The hinge angle $\epsilon_h := \pi/2 - \phi$, and is positive when the $n$-simplices are concave from the given orientation, and negative when convex.

\begin{figure}[h]
\centering
\includegraphics[scale=1]{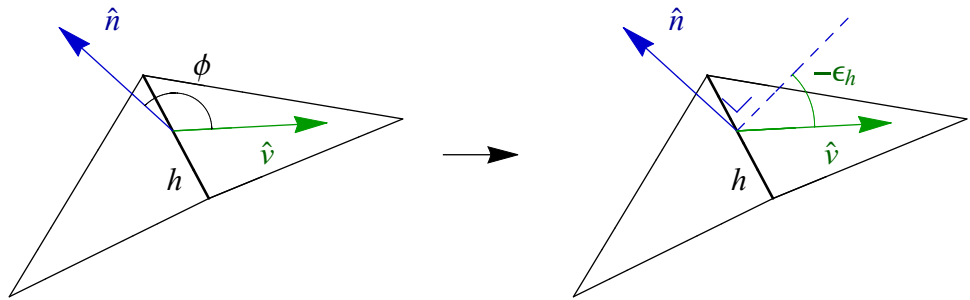}
\caption{Pair of triangles linearly embedded in $\euc{3}$, showing the hinge $h$ and hinge angle $\epsilon_h$.}
\label{fig:hinge}
\end{figure}


A piecewise flat manifold $S^n$ is considered as a piecewise flat \emph{submanifold} of $\euc{n+1}$ if it can be embedded piecewise-linearly so that each simplex is flat in $\euc{n+1}$. The hinge angles $\epsilon_h$ associated with each $(n-1)$-simplex $h$ are determined by the specific embedding of $S^n$ in $\euc{n+1}$, and the embedding itself is determined by these angles, up to Euclidean transformations.

\begin{remark}
Note that the notation for the hinges and hinge angles in this paper coincides with the notation for the co-dimension-two simplices and deficit angles respectively in \cite{PLCurv}. This is done to emphasize the similarity between the piecewise flat extrinsic curvature constructions developed here, and the \emph{intrinsic} curvature constructions in \cite{PLCurv}.
\end{remark}

\subsection{Approximating submanifolds}


A smooth submanifold $M^n \subset \euc{n+1}$ can be approximated by a piecewise flat submanifold $S^n \subset \euc{n+1}$ in a number of ways. Generally, a homogeneous simplicial complex is first defined on $M^n$, with $S^n$ then defined using the same simplicial complex with either the vertices coinciding with those on $M^n$ or the $n$-simplices tangent to $M^n$ at some point in their interior. The piecewise flat submanifold $S^n$ may then be globally rescaled in order to have an equivalent $n$-volume to $M^n$.

In order for $S^n$ to be a good approximation for $M^n$, both the deficit angles and hinge angles should be uniformly small. This can be achieved by defining the simplicial complex on $M^n$ so that the $n$-simplices are small with respect to both the intrinsic and extrinsic curvature of the manifold. This gives a high resolution where either of the curvatures are high, and a lower resolution in regions where $M^n$ is almost flat.

\


\begin{example}[Approximation of smooth manifolds] \
\label{eg:triang}

\

\begin{enumerate}
\item
A regular triangulation of a circle in $\euc{2}$ can be given by first subdividing the circle into $k$ equal length arcs. A regular $k$-sided polygon can then be defined in $\euc{2}$, with edge-lengths of size $|\ell| = 2 \pi r / k$ where $r$ is the radius of the circle. The hinge angles are based at the vertices of the polygon and are of size $\epsilon_h = - 2 \pi/k$.

\item
A regular triangulation of a $2$-sphere in $\euc{3}$ can be defined similarly, using an icosahedron formed by 20 equilateral triangles. For a total area equal to that of a $2$-sphere of radius $r$, 
the length of each edge is $|\ell| = \sqrt{\frac{4 \pi}{5 \sqrt{3}}} \, r$. The hinge angles are based at the edges with an angle of $\epsilon_h = - \arccos [\sqrt{5}/3] \simeq - 0.23228 \pi$.

\item
A cylinder in $\euc{3}$ can be approximated by starting with regular $k$-sided polygons for the cross-sectional circles, set at regular intervals of length $p$ apart, with edges denoted $\ell_a$. The corresponding vertices for neighbouring polygons can then be joined by new line-segments denoted $\ell_b$, of length $p$, forming a surface of flat rectangles. A simplicial manifold is formed by introducing diagonals to each rectangle, denoted $\ell_c$. The length and hinge angle associated with each type of edge is given in table \ref{tab:CylApprox} below.
\end{enumerate}
\end{example}

\

\begin{table}[h]
	\centering
	\begin{tabular}{|c|c|c|}
		\hline
		Edge & $|\ell|$ & $\epsilon_\ell$ \\
		\hline
		$\ell_a$ & $2 \pi r / k$ & $0$ \\
		$\ell_b$ & $p$ & $- 2 \pi / k$ \\
		$\ell_c$ & $\sqrt{|\ell_a|^2 + |\ell_b|^2}$ & 0 \\
		\hline
	\end{tabular}
	\caption{Edge-lengths and hinge angles for a piecewise flat approximation of a cylinder in $\euc{3}$.}
	\label{tab:CylApprox}
\end{table}

\

\begin{figure}[h]
	\centering
	\includegraphics[scale=1]{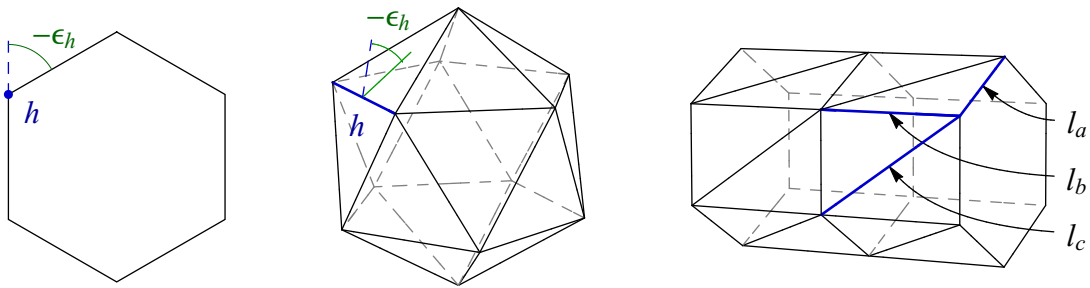}
	\caption{Triangulations of a circle in $\euc{2}$, and of a sphere and cylinder in $\euc{3}$.}
	\label{fig:EgApprox}
\end{figure}

\newpage

\section{Smooth Extrinsic Curvature and Path Integrals}
\label{sec:Int}

\subsection{Smooth extrinsic curvature}


For a smooth submanifold $M^n \subset \euc{n+1}$, the extrinsic curvature at each point $x \in M^n$ gives a measure of how the manifold near $x$ deviates from the tangent space at $x$. This curvature is given by the second fundamental form, a symmetric bilinear form, see for example \cite{KobayNomizuII}. For any pair of vector fields $\vec{u}$ and $\vec{v}$ tangent to $M^n$ in a neighbourhood $U$ of $x$, the second fundamental form is given by the normal component of the change in $\vec{u}$ in the direction of $\vec{v}$:
\begin{equation}
\alpha (\vec{u}, \vec{v})
 := \left< \nabla_{\vec{v}} \vec{u}, \hat n \right> ,
\label{alpha}
\end{equation}
where $\hat n$ is the unit normal vector field to $M^n$ over $U$, $\nabla$ is the covariant derivative associated with the flat connection on $\euc{n+1}$ and $\left< \cdot , \cdot \right>$ is the inner product on $\euc{n+1}$.


Since the normal vector field $\hat n$ is everywhere orthogonal to the vector fields $\vec{u}$ and $\vec{v}$ in $U \subset M^n$, the second fundamental form can also be viewed as the change in $\hat n$ along $\vec{v}$, taking the $\vec{u}$ component of the result:
\begin{equation}
\alpha (\vec{u}, \vec{v}) 
 \equiv - \left< \nabla_{\vec{v}} \hat n, \vec{u} \right>
 =: \left< \mathrm{Q}(\vec{v}), \vec{u} \right> ,
\end{equation}
with $\mathrm{Q}(\vec{v})$ known as the shape operator. The covariant derivative can also be given in terms of the infinitesimal parallel transport of vectors, which is trivial in Euclidean space. For an integral curve $\xi$ of the vector field $\vec{v}$,
\begin{equation}
\nabla_{\vec{v}} \vec{u} |_x 
 = \lim_{\Delta x \rightarrow 0}
 \frac{T^\xi_{(x+\Delta x) \rightarrow x} \vec{u} (x + \Delta x)
 	- \vec{u} (x)}{\Delta x} ,
\end{equation}
with $T^\xi_{x \rightarrow y} \vec{u}$ denoting the parallel transport of the vector $\vec{u}$ at $x$, along the curve $\xi$ to $y$.


All of the second fundamental form values at a point $x \in U$ can be found from those given by a set of $n$ linearly independent vector fields in $U$. These consist of $n$ values for a repeated argument, and $(n-1)/2$ for pairs of different vector fields. Over the neighbourhood $U$ these give a set of $(n+1)/2$ field components, forming a symmetric tensor field commonly referred to as the extrinsic curvature tensor. In the physics literature this tensor is denoted by $\tens{K}$, with $\tens{K}_{a b} \, \vec{u}^a \vec{v}^b := \alpha (\vec{u}, \vec{v})$, with Einstein summation assumed over repeated indices.


The second fundamental form can also be seen as a quadratic form, acting on a single vector field at a time,
\begin{equation}
\alpha(\vec{u})
 := \alpha(\vec{u}, \vec{u}) 
 \equiv \tens{K}_{a b} \, \vec{u}^a \vec{u}^b .
\end{equation}
Values of the bilinear form with mixed arguments can also be given in terms of the quadratic form, using the bilinearity and symmetry of the former. For two vector fields $\vec{u}, \vec{v}$ tangent to $M^n$ in $U \subset M^n$, the action of the bilinear form on the difference vector field $\vec{u} - \vec{v}$ can be expanded,
\begin{eqnarray}
&&\alpha(\vec{u} - \vec{v}, \vec{u} - \vec{v})
= \alpha(\vec{u}, \vec{u})
- 2 \alpha(\vec{u}, \vec{v})
+ \alpha(\vec{v}, \vec{v})
\nonumber \\ \Leftrightarrow \qquad
&&\alpha(\vec{u}, \vec{v})
= \frac{1}{2}\left(
\alpha(\vec{u})
+ \alpha(\vec{v})
- \alpha(\vec{u} - \vec{v})
\right) .
\label{SFFuv}
\end{eqnarray}
The complete second fundamental form at each point $x \in M^n$ can therefore also be given by the quadratic form values for a set of $n$ linearly independent vector fields in $U$, and for the $(n-1)/2$ vector fields given by the difference between each pair of these.

\subsection{Path integrals of curvature along geodesic segments}


Due to the nature of piecewise flat manifolds, infinitesimal application of the second fundamental form will give zero curvature within each $n$-simplex, and an infinite value at the hinges. However, since the hinge angles should be seen as integrated quantities, the path integral of the second fundamental form along smooth manifold geodesics is first investigated below.


\begin{lemma}[Smooth geodesic curvature integral]
\label{lem:PathIntM}
Take a geodesic segment $\gamma (s) : [0, 1] \rightarrow M^n \subset \euc{n+1}$, with unit tangent vectors $\hat \gamma_s \in T_{\gamma(s)} M^n$. The integral along $\gamma$ of the second fundamental form $\alpha (\hat \gamma)$, is equal to the angle between $\hat \gamma_0$ and $\hat \gamma_1$ in $\euc{n+1}$,
\begin{equation}
a_\gamma :=
\int_\gamma \alpha(\hat \gamma) \, \mathrm{d} s
 = \psi  \left(
 \hat \gamma_0, \, T^\gamma_{1 \rightarrow 0} \hat \gamma_1
 \right) .
\label{PathIntM}
\end{equation}
The term $\psi$ represents the angle between its two arguments, and is deemed positive if the normal component of the vector $T^\gamma_{1 \rightarrow 0} \hat \gamma_1 - \hat \gamma_0$ is in the positive $\hat n$ direction, and negative otherwise.
\end{lemma}

\begin{proof}
From the definition of the second fundamental form in (\ref{alpha}),
\begin{equation}
\int_\gamma \alpha(\hat \gamma) \, \mathrm{d} s
= \int_\gamma \left<
 \nabla_{\hat \gamma} \hat \gamma, \, \hat n \right> \, \mathrm{d} s .
\end{equation}
The vector field $\nabla_{\hat \gamma} \hat \gamma$ must be parallel to $\hat n$ at each point of $\gamma$, since it is a geodesic curve in $M^n$ with $\nabla^{M}_{\hat \gamma} \hat \gamma = 0$, where $\nabla^{M}$ is the restriction of $\nabla$ to $M^n$. This derivative can also be given in terms of the parallel transport along $\gamma$, with
\begin{equation}
\nabla_{\hat \gamma} \hat \gamma \, |_{\gamma(s)}
 = \lim_{\Delta s \rightarrow 0} \frac{
   T^\gamma_{(s+\Delta s) \rightarrow s} \hat \gamma_{(s + \Delta s)}
   - \hat \gamma_s
   }{\Delta s} .
\end{equation}
In the limit above, the vector in the numerator can be denoted $d \psi \, \hat n$, where $d \psi$ is the infinitesimal angle of rotation between  $\hat \gamma_s$ and the parallel transport of $\hat \gamma_{(s+d s)}$ back to the point $\gamma(s)$. This infinitesimal angle will be positive if the difference is in the positive $\hat n$ direction, and negative otherwise. The integral of the curvature along $\gamma$ therefore becomes
\begin{equation}
\int_\gamma \alpha(\hat \gamma) \, \mathrm{d} s
 = \int_\gamma \left<
   \nabla_{\hat \gamma} \hat \gamma, \, \hat n \right> \, \mathrm{d} s
 = \int_\gamma \left<
   \frac{d \psi}{d s} \, \hat n, \, \hat n \right> \, \mathrm{d} s
 = \int_\gamma \mathrm{d} \psi
 = \psi  \left(
 \hat \gamma_0, \, T^\gamma_{1 \rightarrow 0} \hat \gamma_1
 \right) ,
\end{equation}
with the finite angle $\psi$ being positive if the normal component of $T^\gamma_{1 \rightarrow 0} \hat \gamma_1 - \hat \gamma_0$ is in the positive $\hat n$ direction, and negative otherwise.
\end{proof}


Applying the result of the lemma above to a hinge $h$ in a piecewise flat manifold $S^n \subset \euc{n+1}$ can easily be seen to give the hinge angle $\epsilon_h$ as the integratal of the curvature along a path intersecting $h$ orthogonally. The right hand side of (\ref{PathIntM}) can even be used to define the hinge angles. More interestingly, lemma \ref{lem:PathIntM} can also be used to define the integrated curvature along paths which are \emph{not} orthogonal to a hinge $h$ in terms of the hinge angle $\epsilon_h$.


\begin{theorem}[Piecewise flat geodesic curvature integral]
\label{thm:PathIntS}
For a geodesic segment $\gamma(s): [0,1] \rightarrow S^n \subset \euc{n+1}$, intersecting a single hinge $h$ at an angle $\theta$ to the vector orthogonal to $h$, the integrated curvature along $\gamma$ is
\begin{equation}
a_h^\theta :=
a_\gamma 
 = \cos \theta \ \epsilon_h + O(\epsilon_h^3) ,
\end{equation}
for a small hinge angle $\epsilon_h$, with the higher order terms vanishing for $\theta \in \{0, \pm \pi/2\}$.
\end{theorem}

\begin{proof}
A geodesic path in $S^n$ will be a straight line in each $n$-simplex that it intersects. The parallel transport of the vector $\hat \gamma_1$, tangent to $\gamma$ at the point $\gamma (1)$, to the point $\gamma (0)$ will have its component orthogonal to $h$ rotated by the hinge angle $\epsilon_h$, with the remaining components unchanged. From figure \ref{fig:PFpath} it is clear that
\begin{equation}
\sin \frac{1}{2} \psi(\hat \gamma_0, \, T^\gamma_{1 \rightarrow 0} \hat \gamma_1)
 = \cos \theta \, \sin \frac{1}{2} \epsilon_h .
\end{equation}
For $\cos \theta \in \{0, 1\}$ the $\sin$'s can be removed from both sides of the equation, otherwise for small hinge angle $\epsilon_h$ the equation reduces to $\psi(\hat \gamma_0, \, T^\gamma_{1 \rightarrow 0} \hat \gamma_1)
= \cos \theta \, \epsilon_h + O(\epsilon_h^3)$.
\end{proof}

\begin{figure}[h]
\centering
\includegraphics[scale=1]{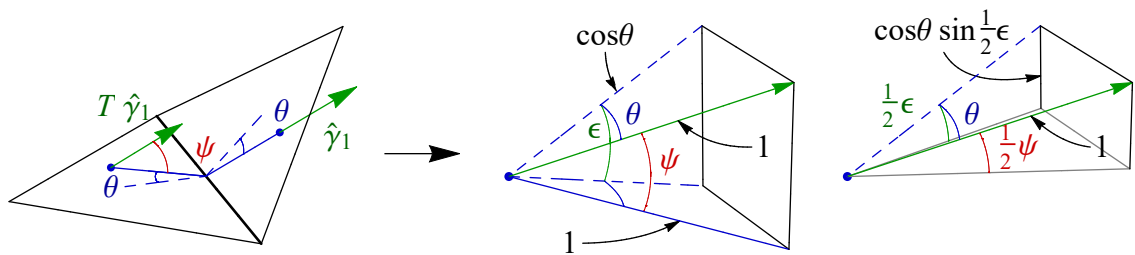}
\caption{The star of a hinge $h$ is shown on the left, with a geodesic curve $\gamma$ and the parallel transport of $\hat \gamma_1$ to $\gamma(0)$. The vectors $\hat \gamma_0$ and $T^\gamma_{1 \rightarrow 0} \hat \gamma_1$ are compared on the right.}
\label{fig:PFpath}
\end{figure}


Since $\lim_{|\gamma| \rightarrow 0} a_\gamma$ gives the infinitesimal curvature $\alpha(\hat \gamma)$ in a smooth manifold, the piecewise flat integrated curvature $a_h^\theta$ can be seen as a discretization with the length of $\gamma$ becoming small with respect to the lattice rather than vanishing. In this sense, small corresponds to intersecting one and only one hinge, though this will be shown to be \emph{too} small in the following sections. Instead these geodesic integrals will be used to give volume integrals over regions intersecting specific collections of hinges.

\section{Mean Curvature}
\label{sec:H}


The mean curvature of a smooth manifold $M^n \subset \euc{n+1}$ is a scalar field on $M^n$, given by the trace of the second fundamental form. For an orthonormal set of vector fields $\{ \vec{e_1}, ... , \vec{e_n} \}$ in a neighbourhood $U \subset M^n$, the mean curvature field over $U$ is
\begin{equation}
H
 := \sum_i^n \alpha (\vec{e_i}, \vec{e_i}) 
  = \sum_i^n \alpha (\vec{e_i}).
\label{Hsmooth}
\end{equation}
This can be used with theorem \ref{thm:PathIntS} to compute the integral of the mean curvature over certain $n$-dimensional subregions of a piecewise flat manifold $S^n$.


\begin{lemma}[Integral of mean curvature over subregion of a hinge star]
\label{lem:IntKh}
Over any $n$-dimensional region $D \subset \mathrm{star}(h) \subset S^n$, the integral of the mean curvature is
\begin{equation}
\int_D H \, \mathrm{d} V^n
= |h_D| \epsilon_h ,
\end{equation}
where $|h_D|$ represents the $(n-1)$-volume of $h \cap D$.
\end{lemma}

\

\begin{proof}
By the definition of the mean curvature, the integral above can be given in terms of the second fundamental form for an orthonormal set of vector fields $\{ \vec{e_1}, ... , \vec{e_n} \}$ in $D$,
\begin{equation}
\int_D H \, \mathrm{d} V^n
 = \int_D \left[ \sum_i^n \alpha(\vec{e_i}) \right] 
  \mathrm{d} V^n
\end{equation}
Since the star of each hinge is intrinsically Euclidean, these vector fields can be chosen to represent a Cartesian coordinate frame, with their integral curves given as geodesic lines within $D \subset S^n$. The integral and sum therefore commute and can be swapped, with the integrals for each value of $i$ above then given in terms of the geodesic curvature integrals,
\begin{equation}
\int_D \left[ \sum_i^n \alpha(\vec{e_i}) \right] 
\mathrm{d} V^n
 = \sum_i^n \left[ \int_D \alpha(\vec{e_i}) \, \mathrm{d} V^n \right]
 = \sum_i^n \left[
   \int_{h_D} a^{\theta_i}_h \, \mathrm{d} V^{n-1} \right] ,
\label{Proof:IntKh2}
\end{equation}
with only the lines intersecting $h_D$ having non-zero integrated curvature. Choosing the vector field $\vec{e_1}$ to be orthogonal to the hinge $h$, the integral curves of the remainder of the vector fields will not intersect $h$. The sum in (\ref{Proof:IntKh2}) can therefore be reduced to a single term, with
\begin{equation}
\int_D H \, \mathrm{d} V^n
 = \int_{h_D} a^{\theta_1 = 0}_h \, \mathrm{d} V^{n-1}
 = \epsilon_h \int_{h_D} \mathrm{d} V^{n-1}
 = |h_D| \epsilon_h ,
\end{equation}
with the hinge angle $\epsilon_h$ invariant over $h$.
\end{proof}

\


This result matches some earlier definitions in the literature \cite{HS81,TraceK}, associating an integrated mean curvature of $|h| \epsilon_h$ with each hinge. Such a definition also leads naturally to Steiner's total mean curvature $\sum_h |h| \epsilon_h$ from \cite{Steiner}. However, unlike the smooth case, hinges have a directionality associated with them for $n > 1$, specifically the direction orthogonal to each hinge in $S^n$. Taking the cylinder in example \ref{eg:triang}, each type of hinge will have a different value of $|h| \epsilon_h$, while a smooth cylinder has the same mean curvature value everywhere.

The mean curvature at each point of a smooth manifold $M^n \subset \euc{n+1}$ can be seen as an average of the second fundamental form over \emph{all} directions in $M^n$. Since regions enclosing single vertices give the most general collection of hinge orientations in a piecewise flat manifold $S^n$, the piecewise flat mean curvature will be given here by integrating over the $n$-dimensional regions of a dual tessellation of $S^n$, similar to \cite{Taub95,MeyDebSchBar03}

\newpage


\begin{definition}[Vertex regions]
\label{def:V_v}
A decomposition of a piecewise flat manifold $S^n$, into $n$-dimensional regions $V_v$ dual to each vertex $v$, is defined so that:
\begin{enumerate}
	\item The vertex $v \in V_v$, but no other vertices are contained within $V_v$.
	
	\item The regions $V_v$ form a complete tessellation of $S^n$,
	\begin{equation}
	|S^n| = \sum_{v \in S^n} |V_v|, \qquad
	V_{v_i} \cap V_{v_j} = \emptyset \quad
	\forall \ i \neq j .
	\end{equation}
	with $|S^n|$ and $|V_v|$ representing the $n$-volumes of $S^n$ and $V_v$ respectively.
	
	\item Only hinges $h_v$ in the star of $v$ intersect $V_v$.
\end{enumerate}
\end{definition}

This gives the most general definition of a dual tessellation for the constructions that follow. The third property is not strictly necessary but is deemed a reasonable requirement, with Voronoi tessellations satisfying this condition only where $S^n$ gives a Delaunay triangulation for example. It also seems reasonable to require an additional property which distributes the total $n$-volume of $S^n$ in some consistent manner over the subregions $V_v$. However there are many different methods for doing this, most of which are incompatible with one another, such as the Voronoi and barycentric tessellations. Such a property will therefore not be imposed here.


\begin{theorem}[Piecewise flat mean curvature]
\label{thm:H_v}
The average mean curvature over a dual $n$-volume $V_n$ is
\begin{equation}
\label{H_v}
H_v
 := \widetilde H_{V_v}
 = \frac{1}{|V_v|}
 \sum_{h \subset \mathrm{star}(v)} |h_{|V_v}| \epsilon_h ,
\end{equation}
with $h_{|V_v} = h \cap V_v$, and $|V_v|$ representing the $n$-volume of $V_v$.
\end{theorem}

\begin{proof}
The region $V_v$ can be decomposed into subregions $D_h$, each enclosing the intersection of the hinge $h$ with $V_v$, and no part of any other hinge. The integral of the mean curvature over each subregion $D_h$ is given by lemma \ref{lem:IntKh}, with the total integral over $V_v$ then given by the sum of these,
\begin{equation}
\int_{V_v} H \ \mathrm{d} V^n
= \sum_{h \subset \mathrm{star}(v)}
\int_{D_h} H \ \mathrm{d} V^n
= \sum_{h \subset \mathrm{star}(v)} |h_{|V_v}| \epsilon_h .
\end{equation}
The average curvature is then found by dividing the integral by the volume of $V_v$.
\end{proof}

\begin{remark}
Instead of the hinge angles, in \cite{Taub95} and \cite{MeyDebSchBar03} the mean curvature comes from the curvature along the edges, essentially given by comparing the tangent vector of the edge with a normal vector constructed at the vertex. These curvatures can be distributed over triangular segments of a Voronoi region $V_v$ associated with each edge. The segments, known as \emph{hybrid cells} in \cite{MMR,TraceK}, are formed by the given vertex and the circumcenters of the triangles on either side of each edge. Integrating over this distribution of curvatures gives precisely the \emph{cotan} formula of \cite{PP93,MeyDebSchBar03, WarDDG, BS07}. This can also be extended to higher dimensions, where the curvature associated with each edge can be computed in the same way, and the volume of these circumcentric hybrid cells is also easily found.
\end{remark}


Unlike other vertex-based mean curvatures, the piecewise flat mean curvature given in theorem \ref{thm:H_v} above still gives the same total mean curvature expression as that of Steiner \cite{Steiner}.

\begin{corollary}[Total mean curvature]
\label{cor:IntH}
The total mean curvature over a piecewise flat manifold $S^n$ is
\begin{equation}
\int_{S^n} H_v \ \mathrm{d} V^n
 = \sum_{h \subset S^n} |h| \epsilon_h .
\label{TotalH}
\end{equation}
\end{corollary}

\begin{proof}
The integral of $H_v$ over $S^n$ is equal to the sum of the integrals for each vertex volume $V_v$. Each part of a given hinge $h$ must be contained in a single volume $V_v$, since the volumes $V_v$ form a tessellation of $S^n$, and since the deficit angles $\epsilon_h$ are fixed over $h$, the contribution from each hinge will be $|h| \epsilon_h$. The total integral is then given by the sum of these terms for all hinges $h \subset S^n$.
\end{proof}

\begin{example}[Piecewise flat mean curvatures]
\label{eg:mean}

\

\begin{enumerate}
	\item For a piecewise linear curve in $\euc{2}$ there is an unambiguous vertex region given by half of each edge in the star of each vertex. For the $k$-sided polygon approximation of a circle, the mean curvature at each vertex is
	\begin{equation}
	H_v^{\mathrm{poly}}
	 = \frac{\epsilon_v}{|\ell|}
	 = \frac{- 2 \pi / k}{2 \pi r / k}
	 = - \frac{1}{r} ,
	\end{equation}
	which is exactly equal to the smooth curvature of a circle in $\euc{2}$.
	
	\item Due to the high level of symmetry in the icosahedron approximation of a $2$-sphere in $\euc{3}$, there is also an unambiguous choice of vertex region $V_v$. This region is formed by the convex hull of the symmetric centres of the equilateral triangles in the star of the vertex. With 12 vertices in an icosahedron, the area of each vertex region $|V_v| = 4 \pi r^2 / 12 = \pi r^2 / 3$, with the mean curvature then given by
	\begin{equation}
	H_v^{\mathrm{icos}}
	 = \frac{5(\frac{1}{2} \, |\ell| \, \epsilon_\ell)}{|V_v|}
	 = \frac{15}{2 \pi r} \, \sqrt{\frac{4 \pi}{5 \sqrt{3}}} \,
	   \left(- \arccos\left[\frac{\sqrt{5}}{3}\right] \right)
	 \simeq - \frac{2.09851}{r} ,
	\end{equation}
	which gives a good approximation for $H = - 2/r$, the mean curvature of a smooth $2$-sphere in $\euc{3}$.
	
	\item The most regular vertex regions in the piecewise flat approximation for the cylinder are given by intrinsically rectangular regions with the vertex at the centre, and sides parallel with the edges $\ell_a$ and $\ell_b$. These give the Voronoi regions around each vertex, and an area of $|V_v| = |\ell_a| \times |\ell_b| = 2 \pi r p / k$. For a vertex region intersecting $q$ diagonal edges $\ell_c$ (which can range from $0$ to $4$), the mean curvature is
	\begin{equation}
	H_v^{\mathrm{cyl}}
	 = \frac{2 \, \frac{1}{2} |\ell_a| \epsilon_a
	 	   + 2 \, \frac{1}{2} |\ell_b| \epsilon_b 
	 	   + q \, \frac{1}{2} |\ell_c| \epsilon_c}{|V_v|}
	 = \frac{- 2 \pi p / k + 0 + 0}{2 \pi r p / k}
	 = - \frac{1}{r} .
	\end{equation}
	This value is constant for all vertices, is invariant to the number of diagonal edges $\ell_c$ intersected, and gives exactly the mean curvature of a smooth cylinder in $\euc{3}$.
\end{enumerate}
\end{example}

\section{Hinge-Orthogonal Component}
\label{sec:Kh}

Due to the structure of the hinge angles, the most natural values of the second fundamental form are those given by vector fields orthogonal to each hinge. However regions enclosing single hinges can be seen to be insufficient, similar to the mean curvature, which can be shown from both the sphere and cylinder triangulations in example \ref{eg:triang}.


The regular shape of the icosahedron approximation of a $2$-sphere means it can be tessellated regularly into regions $D_h$ associated with each hinge. Using part of the arguments from the proof of lemma \ref{lem:IntKh}, the average value of the component orthogonal to a hinge $h$ can then be given by the expression $|h| \epsilon_h / D_h$. However the resulting value differs from the smooth curvature by a factor of about two. Even more problematic are the diagonal edges $\ell_c$ in the cylinder example. These have zero hinge angles, and so the expression $|\ell_c| \epsilon_c / D_c$ vanishes, which is clearly not the case for the corresponding smooth curvature component. The issue with the sphere can be attributed to the fact that there are $n$ linearly independent vectors at each point, while using tessellations assigns each point to a single curvature component. The cylinder problem implies that a region intersecting a larger collection of hinges is required.



A plausible solution is to follow the mean curvature approach, but to use a union of the vertex regions $V_v$ for the vertices $v$ in the closure of each hinge. Unfortunately the vectors orthogonal to a given hinge cannot be defined over the regions $V_v$ in an unambiguous way. A new set of $n$-dimensional regions is therefore defined as a subspace of this union, where vectors can be parallel transported unambiguously, and where the integrated curvature components are stable to small deviations of hinges across the boundaries internal to each vertex region.



\begin{definition}[Hinge regions]
\label{def:Vh}
The $n$-dimensional regions $V_h$ associated with each hinge $h$ are defined as the union of the regions $V_v$ dual to the vertices $v$ in the closure of $h$, intersected with the geodesic extensions $\gamma^\perp_h$ of the lines orthogonal to $h$,
\begin{equation}
V_h
 := \left(\cup_v \, V_v\right)
 \cap \int_h \gamma^\perp_h \, \mathrm{d} V^{n-1} ,
    \quad \text{such that} \quad
	v \in \bar h .
\end{equation}
\end{definition}


Each region $V_h$ contains all of the hinge $h$, and parts of other hinges, and will overlap with regions of these other hinges. The average curvature orthogonal to each hinge $h$ is now given by integrating over these hinge regions.



\begin{theorem}[Piecewise flat hinge-orthogonal curvature]
\label{thm:alpha_h}
In a piecewise flat manifold $S^n \subset \euc{n+1}$, the average curvature component orthogonal to a hinge $h$ over the region $V_h$ is
\begin{equation}
\label{alpha_h}
\alpha_h
 := \widetilde{\alpha}(\hat \gamma^\perp_h)_{V_h}
 = \frac{1}{|V_h|} \left(
	|h| \epsilon_h
	+ \sum_{i} |h_i \cap V_h| \, \cos^2 \theta_i \, \epsilon_i
	+ O(\epsilon_i^3)
	\right) ,
\end{equation}
with the sum taken over the other hinges $h_i$ intersecting $V_h$, with $\epsilon_i$ representing the hinge angle at $h_i$, and $\theta_i$ the angle between $\gamma^\perp_h$ and the lines orthogonal to each hinge $h_i$.
\end{theorem}

\begin{figure}[h]
\centering
\includegraphics[scale=1]{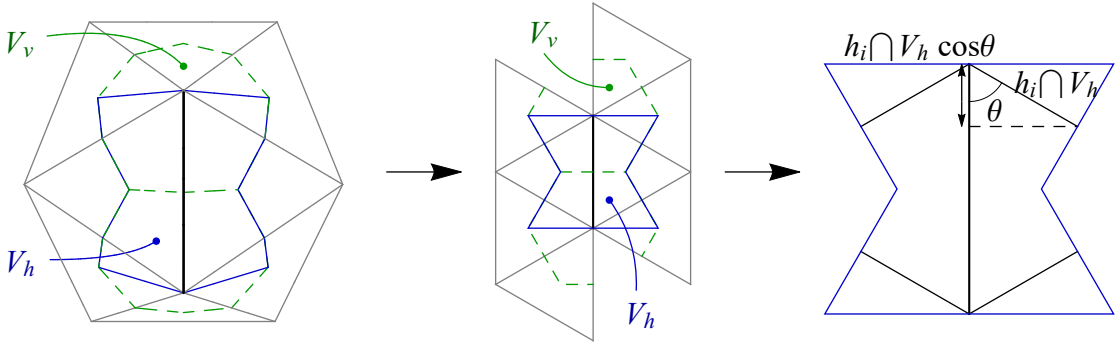}
\caption{Hinge region for an icosahedron in $\euc{3}$.}
\label{fig:IIh}
\end{figure}

\begin{proof}
For each geodesic segment $\gamma := \gamma^\perp_h \ \cap \ V_h$, the path integral of the second fundamental form $\alpha(\hat \gamma)$ along $\gamma$ is given by the sum of the contributions from each hinge it intersects,
\begin{equation}
  a_{\gamma}
 = a_h + \sum_i a^{\theta_i}_{h_i}
 = \epsilon_h
 + \sum_i \cos \theta_i \, \epsilon_i
 + O(\epsilon_i^3) .
\end{equation}

\

\noindent
The integral of $\alpha(\hat \gamma)$ over the entire region $V_h$ is then given by integrating these path integrals over all of $h$, which can also be given by a sum of contributions from each hinge,
\begin{equation}
\int_{V_h} \alpha(\hat \gamma) \, \mathrm{d} V^n
 = \int_h
   a_{\gamma} \, \mathrm{d} V^{n-1}
 = |h| \epsilon_h
 + \sum_i \int_h a^{\theta_i}_{h_i} \, \mathrm{d} V^{n-1} .
\label{alpha_h_2}
\end{equation}
For each hinge $h_i$, the cross-sectional $(n-1)$-volume of the geodesic segments $\gamma$ which intersect $h_i$ is $|h_i \cap V_h| \cos \theta_i$, see figure \ref{fig:IIh} for example. Since $\theta_i$ and $\epsilon_i$ are constant over each hinge $h_i$, the final expression for the integral of $\alpha(\hat \gamma)$ over $V_h$ is
\begin{equation}
\int_{V_h} \alpha(\hat \gamma) \mathrm{d} V^n
 = |h| \epsilon_h
 + \sum_i |h_i \cap V_h| \cos^2 \theta_i \, \epsilon_i 
 + O(\epsilon_i^3) .
\end{equation}
The average value is then given by dividing by the $n$-volume of the region $V_h$.
\end{proof}

\


For a two dimensional piecewise flat manifold $S^2 \subset \euc{3}$, the second fundamental form can be completely determined within each triangle, using the curvature orthogonal to each of the edges (hinges) $\ell \subset S^2$. For any given triangle, with edges denoted $\ell_1$, $\ell_2$ and $\ell_3$,
\begin{equation}
   |\ell_1| \, \hat \ell^\perp_1
 + |\ell_2| \, \hat \ell^\perp_2
 + |\ell_3| \, \hat \ell^\perp_3
 = 0 .
\end{equation}
Taking the normal vectors to two of the sides within $S^2$ as basis vectors, the relation above can be used instead of the difference vector in equation (\ref{SFFuv}) to give the second fundamental biliear form for mixed arguments as
\begin{eqnarray}
\alpha(\hat \ell^\perp_1, \hat \ell^\perp_2)
 &=& \frac{1}{|\ell_1| \, |\ell_2|}\left(
 - \alpha(|\ell_1| \, \hat \ell^\perp_1)
 - \alpha(|\ell_2| \, \hat \ell^\perp_2)
 + \alpha(|\ell_1| \, \hat \ell^\perp_1 + |\ell_2| \, \hat \ell^\perp_2)
 \right) \nonumber \\
 &=& \frac{1}{|\ell_1| \, |\ell_2|}\left(
 - |\ell_1|^2 \, \alpha_{\ell_1}
 - |\ell_2|^2 \, \alpha_{\ell_2}
 + |\ell_3|^2 \, \alpha_{\ell_3}
 \right) .
\end{eqnarray}
This gives a complete extrinsic curvature tensor within each triangle, using only the hinge-orthogonal curvatures.

\ \\

\begin{example}[Hinge-orthogonal piecewise flat curvature components] \
\label{eg:mean}

\begin{enumerate}
\item For a piecewise linear curve in $\euc{2}$ the curvature component $\alpha_h \equiv H_v$. In the case of the $k$-sided regular polygon, this gives a curvature component of $\alpha = - 1/r$ which is equal to the smooth curvature of a circle in $\euc{2}$.

\item The hinge region $V_h$ for the icosahedron consists of two thirds of the area of each of the equilateral triangles on either side of the hinge (edge) $h$, and a sixth of the four triangles bounding these, see figure \ref{fig:IIhEg}, so that $|V_h| = 2 \times 4 \pi r^2 / 20 = 2 \pi r^2 / 5$. There are four other edges with half of their lengths intersecting this volume, each making an angle of $\theta = \pi/3$ with $h$. The curvature component orthogonal to each hinge is therefore
\begin{equation}
\alpha_h
 = \frac{|\ell| \epsilon
   + 4\left(\frac{1}{2} |\ell| \cos^2 \frac{\pi}{3} \, \epsilon\right)
   }{|V_h|}
 = \frac{15}{4 \pi r} \, \sqrt{\frac{4 \pi}{5 \sqrt{3}}} \,
   \left( - \arccos\left[\frac{\sqrt{5}}{3}\right] \right)
 \simeq - \frac{1.04926}{r} ,
\end{equation}
which approximates the curvature components of a smooth $2$-sphere, $\alpha(\vec{\hat u}) = - 1/r$, to the same level of accuracy as the mean curvature.

\item Since the piecewise flat approximation of the cylinder is intrinsically flat, with zero deficit angles around each vertex, the hinge regions will consist of half of the areas of each of the two vertex regions, so $|V_a| = |V_b| = |V_c| = 2 \pi r p / k$. For the hinge regions $V_a$ and $V_b$, there are edges $\ell_b$ and $\ell_a$ on the boundaries, with $q$ representing the number of diagonal edges $\ell_c$ (from zero to four).

\begin{figure}[h!]
	\centering
	\includegraphics[scale=1]{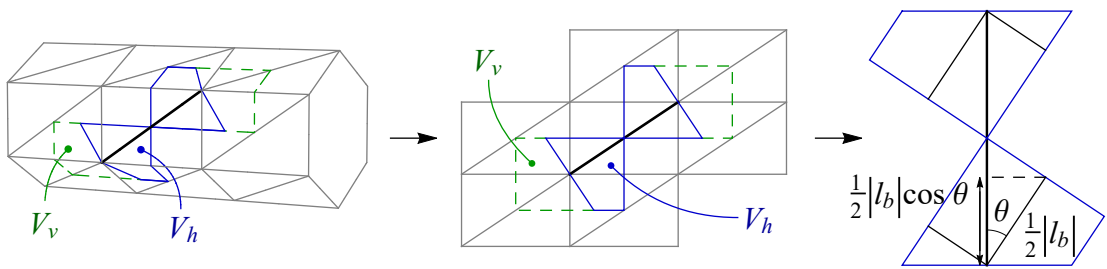}
	\caption{Hinge region for cylinder triangulation edge $\ell_c$.}
	\label{fig:IIhEg}
\end{figure}

The angles between each pair of edges are $\theta_{ab} = \pi/2$, $\theta_{ac} = \arctan\left[|\ell_b| / |\ell_a|\right]$ and $\theta_{bc} =\arctan\left[|\ell_a| / |\ell_b|\right]$, and the curvature components orthogonal to each type of hinge are
\begin{align}
\alpha_a
&= \frac{|\ell_a| \epsilon_a
	+ 2 \frac{1}{2} |\ell_b| \cos^2 \frac{\pi}{2} \, \epsilon_b
	+ q \frac{1}{2} |\ell_c| \cos^2 \theta_{a c} \, \epsilon_c
	}{|V_a|}
 = 0 ,
\\
\alpha_b
&= \frac{|\ell_b| \epsilon_b
	+ 2 \frac{1}{2} |\ell_a| \cos^2 \frac{\pi}{2} \, \epsilon_a
	+ q \frac{1}{2} |\ell_c| \cos^2 \theta_{b c} \, \epsilon_c
    }{|V_b|}
 = - \frac{1}{r} ,
\\
\alpha_c
&= \frac{|\ell_c| \epsilon_c
	+ 2 \frac{1}{2} |\ell_a| \cos^2 \theta_{a c} \, \epsilon_a
	+ 2 \frac{1}{2} |\ell_b| \cos^2 \theta_{b c} \, \epsilon_b
	}{|V_a|}
 = - \frac{\cos^2 \theta_{b c}}{r} .
\end{align}
These exactly match the smooth values, where the curvature component for a unit vector $\vec{\hat u}$ tangent to a cylinder, making an angle $\theta$ with a circular cross-section in $\euc{3}$, is $\alpha(\vec{\hat u}) = - \cos^2 \theta / r$.
\end{enumerate}
\end{example}

\section{Generalized Embeddings}
\label{sec:GenEmb}





For a smooth manifold $M^n \subset \euc{n+m}$ the second fundamental form is valued in the normal tangent bundle to $M^n$ in $\euc{n+m}$. For $m = 1$, $M^n$ is a hypermanifold and the normal bundle is one dimensional, so the scalar values of the magnitude is all that is required to determine the extrinsic curvature. For $m > 1$, the second fundamental form is instead defined as
\begin{equation}
\alpha (\vec{u}, \vec{v})
 := \left< \nabla_{\vec{v}} \vec{u}, \hat n_{\alpha} \right> \hat n_\alpha ,
\end{equation}
where $\hat n_\alpha$ is the unit normal vector to $M^n$ which gives the projection of $\nabla_{\vec{v}} \vec{u}$ into the normal bundle. As can be seen by the Serret-Frenet frame for a curve in $\euc{3}$, see for example \cite{SulDDG}, the curvature is no longer sufficient to describe the embedding of $M^n$ into $\euc{n+m}$ completely, with information about the variation in $\hat n_\alpha$ over $M^n$ also required.

Not all piecewise flat manifolds can be embedded in a Euclidean space of only one extra dimension, but can be embedded if the dimension of the Euclidean space is increased. For an embedding of $S^n$ into $\euc{n+m}$, with $m > 1$, the normal vector to each $n$-simplex is no longer unique. However for each hinge $h$ there is an unambiguous subspace $\euc{n+1} \subset \euc{n+m}$ containing the two $n$-simplices on either side of $h$, where the angle $\epsilon_h$ can be defined. The hinge angle can then be multiplied by the unit vector $\hat n_h$, normal to $h$ in $\euc{n+1}$, which makes equal angles with the normal vectors to each $n$-simplex in the star of $h$. Since these $(n+1)$-dimensional spaces will not be consistent across all hinges, the relationship between them is also required for a complete description of the embedding, as with the smooth case.

The integral of the \emph{magnitude} of the second fundamental form, tangent to a geodesic curve $\gamma \subset M^n \subset \euc{n+m}$, will still be given by $a_\gamma$ from lemma \ref{lem:PathIntM}. As a result, the integral of the magnitude of the piecewise flat curvature across a hinge $h \subset S^n \subset \euc{n+m}$ will still be given by $a_h^\theta$ from theorem \ref{thm:PathIntS}. Equations (\ref{H_v}) and (\ref{alpha_h}) then give the average magnitudes of the mean and hinge-orthogonal curvatures. Orientations for these curvatures could then be determined by a weighted average of the unit normal vectors at each hinge. However the details should depend on a complete piecewise flat version of the smooth variation of the unit normal vectors $\hat n_\alpha$.




\section{Conclusion}
\label{sec:Con}


Expressions have been given for the piecewise flat mean curvature at each vertex (\ref{H_v}) and the extrinsic curvature components orthogonal to the hinges (\ref{alpha_h}) of a piecewise flat manifold $S^n \subset \euc{n+1}$, and shown to give good approximations for triangulations of a circle in $\euc{2}$ and a sphere and cylinder in $\euc{3}$. Since these expressions depend on a collection of hinges, they should be stable to different triangulations of the same smooth manifold, as long as the triangulations give small deficit and hinge angles everywhere. The definitions are also not dependent on any individual choice of dual tessellation, as long as the properties of definition \ref{def:V_v} are satisfied, and can be used in any dimension.



\section*{Acknowledgements}

I'd like to thank Warner A. Miller, Maximilian Hanush and Christopher Duston for many helpful discussions. 


\bibliography{Ref}

\begin{thebibliography}{10}

\bibitem{Steiner}
J~Steiner.
\newblock {\"{U}}ber parallel fl{\"{a}}chen.
\newblock {\em (Monats-) Bericht Akad. Wiss. Berlin}, 1840.

\bibitem{Regge}
T~Regge.
\newblock General relativity without coordinates.
\newblock {\em Nuovo Cimento}, {\bf 19}, 558, 1961.

\bibitem{PLCurv}
R~Conboye and W~A Miller.
\newblock Piecewise flat curvature and {Ricci} flow in three dimensions.
\newblock {\em To appear in The Asian Journal of Mathematics,
  arXiv:1603.03113}, 2016.

\bibitem{PP93}
U~Pinkall and K~Polthier.
\newblock Computing discrete minimal surfaces and their conjugates.
\newblock {\em Experim. Math.}, {\bf 2} 15-36, 1993.

\bibitem{MeyDebSchBar03}
M~Meyer, M~Desbrun, P~Schr{\"{o}}der, and A~H Barr.
\newblock Discrete differential-geometry operators for triangulated
  2-manifolds.
\newblock In H-C Hege and K~Polthier, editors, {\em Visualization and
  Mathematics III}, pages 35--57. Springer, 2003.

\bibitem{WarDDG}
M~Wardetzky.
\newblock Convergence of the cotangent formula: An overview.
\newblock In A~I Bobenko, P~Schr{\"{o}}der, J~M Sullivan, and G~M Ziegler,
  editors, {\em Discrete Differential Geometry}, pages 275--324.
  Birkh{\"{a}}user, 2008.

\bibitem{BS07}
A~I Bobenko and A~Springborn.
\newblock A discrete {Laplace}-{Beltrami} operator for simplicial surfaces.
\newblock {\em Discrete Comput Geom}, {\bf 38}, 4, 740, 2007.

\bibitem{BDGintrins}
J-D Boissonnat, R~Dyer, and A~Ghosh.
\newblock Constructing intrinsic delaunay triangulations of submanifolds.
\newblock {\em Research Report RR-8273, INRIA, arXiv:1303.6493}, 2013.

\bibitem{Grin06}
E~Grinspun, Y~Gingold, J~Reisman, and D~Zorin.
\newblock Computing discrete shape operators on general meshes.
\newblock {\em Eurographics (Computer Graphics Forum)}, {\bf 25} 547-556, 2006.

\bibitem{BobCurv10}
A~I Bobenko, Pottman H, and Wallner J.
\newblock A curvature theory for discrete surfaces based on mesh parallelity.
\newblock {\em Math. Ann.}, {\bf 348}, 1-24, 2010.

\bibitem{HS81}
J~B Hartle and R~Sorkin.
\newblock Boundary terms in the action for the {Regge} calculus.
\newblock {\em Gen. Rel. Grav.}, {\bf 13}, 541-9, 1981.

\bibitem{Brewin}
L~Brewin.
\newblock The {Riemann} and extrinsic curvature tensors in the {Regge}
  calculus.
\newblock {\em Class. Quantum Grav.}, {\bf 5}, 1193, 1988.

\bibitem{KLM89ec}
A~Kheyfets, N~J LaFave, and W~A Miller.
\newblock Pseudo-{Riemannian} geometry on a simplicial lattice and the
  extrinsic curvature tensor.
\newblock {\em Phys. Rev. D}, {\bf 39}, 04 1097, 1989.

\bibitem{TraceK}
R~Conboye, W~A Miller, and S~Ray.
\newblock Distributed mean curvature on a discrete manifold for {Regge}
  calculus.
\newblock {\em Class. Quantum Grav.}, {\bf 32} 185009, 2015.

\bibitem{SulDDG}
J~M Sullivan.
\newblock Curvatures of smooth and discrete surfaces.
\newblock In A~I Bobenko, P~Schr{\"{o}}der, J~M Sullivan, and G~M Ziegler,
  editors, {\em Discrete Differential Geometry}, pages 175--188.
  Birkh{\"{a}}user, 2008.

\bibitem{KobayNomizuII}
S~Kobayashi and K~Nomizu.
\newblock {\em Foundations of Differential Geometry II}.
\newblock Interscience Publishers, John Wiley \& Sons, 1969.

\bibitem{Taub95}
G~Taubin.
\newblock Estimating the tensor of curvature of a surface from a polyhedral
  approximation.
\newblock {\em Proc. of International Conference on Computer Vision}, 902-907,
  1995.

\bibitem{MMR}
J~R McDonald and W~A Miller.
\newblock A geometric construction of the {Riemann} scalar curvature in {Regge}
  calculus.
\newblock {\em Class. Quantum Grav.}, 25 195017, 2008.

\end{thebibliography}
\bibliographystyle{unsrt}

\end{document}